\documentclass[12pt]{article} 

\usepackage{graphicx}
\usepackage{amsmath,amssymb,textcomp,graphicx,amsthm}  
\usepackage{epsfig}
\usepackage{color}
\usepackage[english]{babel} 
\usepackage[alphabetic]{amsrefs}
 \usepackage{verbatim}
\usepackage[latin1]{inputenc}  
\usepackage[allcolors = blue, colorlinks = true]{hyperref}
\usepackage{enumerate}

\numberwithin{equation}{section}

\newcommand{\e}{\varepsilon}
\newcommand{\eps}{\varepsilon}
\newcommand{\<}{\langle}
\renewcommand{\>}{\rangle}

\newtheorem{thm}{Theorem}
\newtheorem{de}[thm]{Definition}
\newtheorem{lem}[thm]{Lemma}
\newtheorem{prop}[thm]{Proposition}
\newtheorem{cor}[thm]{Corollary}

\newcommand{\dd}{\partial}
\renewcommand{\d}{\hspace{0.05em}\textup{d}}

\newcommand{\dist}{\operatorname{dist}}

\renewcommand{\div}{\operatorname{div}}

\newcommand{\ont}{\int\limits_0^T\!\!\!\int\limits_\Omega}

\newcommand{\Om}{\Omega}
\newcommand{\om}{\omega}

\title{Nonnegative Viscosity and Weak  Solutions to Trudinger's Parabolic Equation }

\author{Peter Lindqvist, Mikko Parviainen, Jarkko Siltakoski}
\date{\phantom{abc}}
\begin{document}

\maketitle

\begin{abstract}
\noindent We study the relation between viscosity solutions and weak solutions of the equation 
$$
\dd_t\left(|u|^{p-2}u\right)\,=\, \div(|\nabla u|^{p-2}\nabla u).
$$ We prove that at least for non-negative functions these are equivalent concepts.
\end{abstract}

\bigskip 

{\small \textsf{AMS Classification 2020}: 35D40, 35D30, 35K65, 35K92} 


\section{Introduction}
We consider the doubly non-linear equation
\begin{equation} \label{trudinger}
\frac{\dd\,(|u|^{p-2}u)}{\dd t}\,=\, \div (|\nabla u|^{p-2}\nabla u)   \quad \text{in }\quad \Omega\times (0,T).
\end{equation}
N. Trudinger introduced the equation in \cite{T} mainly to illustrate a Harnack Inequality of the same type as for the ordinary Heat Equation. The latter corresponds to $p=2$. No intrinsic scaling is needed. The equation appears in a much studied \emph{non-linear eigenvalue problem}: if $\psi$ is a (weak) solution to
\begin{equation*}\begin{cases}
    \Delta_p\psi + \lambda|\psi|^{p-2}\psi\,=\,0\quad \text{in}\quad \Omega\\
    \psi \in W^{1,p}_0(\Omega),
\end{cases}\end{equation*}
where $\Delta_p\psi = \nabla \cdot \bigl(|\nabla \psi|^{p-2}\nabla \psi \bigr)$,
then
$$u(x,t)\,=\,e^{-\lambda t/(p-1)}\psi(x)$$
is a solution of the Trudinger Equation. See \cite{L}.

    The theory for sign-changing solutions is less developed than that for the Evolutionary $p$-Laplace Equation
    \begin{equation} \label{pparab}
    \frac{\partial u}{\partial t} \,=\, \div(|\nabla u|^{p-2}\nabla u).
\end{equation}
A decisive difference between the equations (\ref{trudinger}) and (\ref{pparab}) is the degeneracy when $u=0$.
The zeros are of essence even for non-negative solutions of the Trudinger Equation. They are at the fore of our work.

We aim at investigating the relationship between two different concepts of solutions:
\begin{itemize}
\item \textsf{Weak Solutions.} By definition, they belong to the Sobolev space\\ $L^p_{loc}\bigl(0,T;W^{1,p}_{loc}(\Omega)\bigr)$ and satisfy an integral identity with test functions under the integral sign. (Definition \ref{solution})
  \item \textsf{Viscosity Solutions.} They are continuous functions which are tested pointwise with touching test functions. (Definition \ref{viscdef})
\end{itemize}
Our exposition is written only for the case $p>2$.  Our main result is for non-negative viscosity solutions. Thus zeros are allowed.

\begin{thm}\label{thmI} Let $p \geq 2.$ A non-negative viscosity solution in $\Omega_T$ is also a weak solution and as such it belongs to $L^p_{loc}\bigl(0,T;W^{1,p}_{loc}(\Omega)\bigr).$
  \end{thm}

  \medskip
  This relies heavily on our recent result in \cite{LPS} where we proved the theorem for \emph{strictly} positive viscosity solutions. The introduced zero set requires some sophisticated arrangements with so-called barriers. In the opposite direction, we have a result, valid even for sign-changing weak solutions.

  \begin{thm}\label{thmII} Let $p\geq 2.$ A weak solution in $\Omega_T$ is also a viscosity solution, be it sign-changing or not.
  \end{thm}

  \medskip
  The proof is based on a recently achieved version of the Comparison Principle, cf. \cite{ALP}, which fortunately is sufficient for our purpose.
For non-negative weak solutions, Theorem \ref{thmII} was proved by Lindgren and Lindqvist in \cite[Theorem~3]{LL}, and for strictly positive solutions the equivalence of weak and viscosity solutions was proven in \cite[Theorem~2.9]{LPS}. Bhattacharya and Marazzi \cite{BM} proved for example a comparison principle for strictly positive viscosity solutions. Viscosity solutions of a different doubly nonlinear equation were studied by Hynd and Lindgren in \cite{HL}.

\section {Definitions and Known Properties}

We use standard notation. Let $\Omega_T\,=\,\Omega \times (0,T)$, where $\Omega \subset \mathbb{R}^n$ is an open bounded domain with Lipschitz regular \emph{parabolic boundary}
$$\dd_{p}\Omega_T\,=\, (\overline{\Omega} \times \{0\})\,\cup\,(\dd \Omega \times [0,T]).$$

Occasionally, we employ the \emph{Stekloff average}
$$[f(x,t)]_h\,=\, \frac{1}{h}\int\limits_{t-h}^{t}f(x,\tau)\,\d \tau,\qquad h > 0,$$
for regularizing the time derivative.

\begin{de}\label{solution} We say that $u \in L_{loc}^p(0,T;W_{loc}^{1,p}(\Omega))$ is a \emph{weak subsolution}   in $\Omega_T$ if
  \begin{equation}\label{weaksub}\ont\left(-|u|^{p-2}u\,\phi_t\,+\,\<|\nabla u|^{p-2}\nabla u,\nabla \phi\> \right)\d x\d t\,\,\leq\,\,0\end{equation}
  for all $\phi \in C_0^{\infty}(\Omega_T),\ \phi \geq 0.$
  
 We say that $v \in L_{loc}^p(0,T;W_{loc}^{1,p}(\Omega))$ is a \emph{weak supersolution}   in $\Omega_T$ if
  \begin{equation}\label{weaksuper}\ont\left(-|v|^{p-2}v\,\phi_t\,+\,\<|\nabla v|^{p-2}\nabla v,\nabla \phi\> \right)\d x\d t\,\,\geq\,\,0\end{equation}
 for all $\phi \in C_0^{\infty}(\Omega_T), \ \phi \geq 0.$

 A \emph{weak solution} is both a  weak subsolution and supersolution.
\end{de}
Weak solutions have locally H\"{o}lder continuous representatives by \cite{BDL}. Weak supersolutions which are locally essentially bounded below admit lower semicontinuous representatives, cf.\ \cite[Theorem~2.3]{Liao}. In this work, we treat only \emph{continuous} super- and subsolutions. That is not a serious restriction.
For test functions $\zeta\,\geq\,0$ vanishing on $\partial\Omega \times[t_1,t_2]$, but not necessarily at the times $t_1$ and $t_2$,
the weak formulation gives
\begin{equation}\label{zetaeq}
   -\bigg\vert_{t_1}^{t_2}\int\limits_{\Omega} \zeta |u|^{p-2}u \d x \, \,\geq\,
    \int\limits_{t_1}^{t_2}\!\!\!\int\limits_\Omega\left(-|u|^{p-2}u\,\zeta_t+\<|\nabla u|^{p-2}\nabla u,\nabla \zeta\> \right)\d x\d t
\end{equation}
for a weak subsolution $u \in L_{loc}^p(0,T;W^{1,p}(\Omega)),\,\,\, 0 < t_1 < t_2 < T,$ and the above notation means substitution at times $t_1$ and $t_2$.

A continuous weak solution obeys the \emph{Maximum Principle}, cf.\ \cite[Section~2]{LPSarsa}:
$$\min_{\dd_p \Omega_T} \{u\}\,\leq\,u\,\leq \,\max_{\dd_p \Omega_T}\{u\}\quad \text{if}\quad u \in C(\overline{\Omega_T}).$$

In this paper, we prove directly by barriers that a non-negative
viscosity supersolution $u$ in $\Omega_T$ which vanishes at an
interior point $(x_0,t_0)$ satisfies
$$u(x,t)\,=\,0\quad\text{in}\quad\Omega\times(0,t_0).$$
This property is a step in our proof of the equivalence of weak
and viscosity solutions. A similar property is already known for the weak supersolution. The corresponding \emph{strong minimum
principle} is not generally valid at nonzero minimum levels.
Zero has a special status!

We need the following special version of the Comparison Principle, valid when there is a gap on the parabolic boundary $\partial_p\Omega_T$  between the functions.
\begin{thm}[\textsf{Gap comparison}] \label{gap} Let $p \geq 2$. Suppose that $u,v \in C(\overline{\Omega}_T)$ and
\[
u,v \in L^p(0,T;W^{1,p}(\Omega)).
\]
If $v$ is a weak supersolution and $u$ a weak subsolution such that $$\min_{\partial_p \Omega_T}\,\{v-u\}\, > \, 0,$$
  then $v\geq u$ in $\Omega_T$.
\end{thm}

\begin{proof}
The proof of \cite[Theorem~3.1]{ALP} applies also to weak sub- and supersolutions, since the test functions have the required signs.
\end{proof}

\begin{de}\label{viscdef} Let $p>2$. We say that a continuous function $v\,:\Omega_T \rightarrow \mathbb{R}$ is a \emph{viscosity supersolution} in $\Omega_T$, if the following is valid whenever a function $\phi \in C^2(\Omega_T)$ touches $v$ from below at a point $(x_0,t_0)$:
  $$\partial_t(|\phi(x_0,t_0)|^{p-2}\phi(x_0,t_0))\,-\,\Delta_p\phi(x_0,t_0)\,\geq\,0.$$
  \emph {Viscosity subsolutions} have a similar definition, but with test functions touching from above and the inequality reversed. The {viscosity solutions} are required to be both viscosity sub- and supersolutions.
\end{de}

\noindent\textbf{Remark.}
\begin{enumerate}

\item We may assume strict touching: $\phi(x,t) < v(x,t)$ except at $(x,t) = (x_0,t_0)$.

\item The common requirement
$$\inf\{v-\phi\}\,=\, v(x_0,t_0) - \phi(x_0,t_0)$$
does not work, when $p \neq 2$, since the point value appears in the above definition.
\end{enumerate}

\medskip
\paragraph{Some Barriers.}  We shall need a few expedient classical solutions, suitable in connexion with the Comparison Principle. They are called \emph{barriers}.

In $B_R \times (0,\infty)$, the function
$$V(r,t)\,=\,\exp\left(-\frac{\lambda t}{p-1}\right)\,\psi_{\lambda}(r),\quad r =|x|,$$
is of class $C^2$ away from the line $x=0,\,0<t<\infty,$ and it is of class $C^1$ in the whole domain. Here $\psi_{\lambda} > 0$ solves the   equation
\begin{equation}\label{eigenfunction}
\Delta_p\psi_{\lambda} + \lambda |\psi_{\lambda}|^{p-2}\psi_{\lambda}\,=\,0,\quad\psi_{\lambda}\Big\vert_{\partial B _R}\,=\,0,\quad \psi_{\lambda} > 0,
\end{equation}
and $\lambda = \lambda(R)$ is ``the first eigenvalue of the $p$-Laplace operator'' in a ball, see \cite{B}. 
In fact\footnote{\samepage Near the origin, $\psi_{\lambda}(r) = \psi_{\lambda}(0)- ar^{p/(p-1)} + o(r^{p/(p-1)})$, where $a>0$. Indeed, integrating the radial equation gives
$$r^{n-1}|\psi'_{\lambda}(r)|^{p-2}\psi'_{\lambda}(r)
=-\lambda\int_0^r s^{n-1}\psi_{\lambda}(s)^{p-1}\,ds.$$
Hence
$$\psi'_{\lambda}(r)
=-\left(\frac{\lambda}{n}\right)^{1/(p-1)}\psi_{\lambda}(0)\,r^{1/(p-1)}
+o(r^{1/(p-1)}).$$
Integrating once more, we obtain the claimed formula. For $p>2$, this example is $C^{1,\alpha}$ but not $C^2$ at the origin, with $\alpha=1/(p-1)$, showing that solutions to Trudinger's Equation are $C^{1,\alpha}$ only.}, writing the $p$-Laplacian in the radial coordinates in this case, we obtain
$$\frac{\partial\,\,}{\partial r}\Bigl(r^{n-1}|\psi'_{\lambda}|^{p-2}\psi'_{\lambda}\Bigr) + \lambda r^{n-1}|\psi_{\lambda}|^{p-2}\psi_{\lambda}\,=\,0,\quad \psi_{\lambda}(R)=0,\,\psi'_{\lambda}(0) =0.$$
Since $r^{n-1}|\psi'_{\lambda}(r)|^{p-1}=O(r^n)$ as $r \to 0$, the inner boundary term in integration by parts over $B_R\setminus\overline{B_r}$ tends to zero. Thus the eigenvalue equation holds weakly also at the origin, and $V$ is a weak solution to Trudinger's Equation.

\begin{lem}\label{omega} Let $p>2$ and $0<R<1.$ In $(B_1\backslash\overline{B_R})\times (0,\tau)$, the function
  $$\omega(x,t)\,=\,(1-|x|)\exp\left(-|x|t^{-\frac{1}{p-1}}\right)$$
  is a classical subsolution, where\footnote{The provided bound is useful, but irrelevant. When $n=1,\,\tau = \infty.$ For $p=2$ and $n>1$, one may take $\tau=R^2/(n-1)^2$.}
  $$\tau\,=\,\Bigl(\tfrac{p-2}{n-1}R\Bigr)^{p-1}\quad (n>1)$$
  will do.
  Furthermore, $\Delta_p\,\omega > 0.$
\end{lem}

\begin{proof}
We abbreviate 
$$\om(x,t)\,=\, (1-r)\exp\left(-rt^{-1/(p-1)}\right)\,=\, (1-r)\Psi(r,t),\quad r =|x|,$$  and calculate, denoting $\,'=\frac{\partial\,}{\partial r}$,
\begin{align*}
  \partial_t\,\om\,&=\, \Psi(r,t)(1-r)r\Bigl(\tfrac{1}{p-1}t^{-\,p/(p-1)}\Bigr)\\
    \om'\,&=\,\Psi(r,t)\Bigl(-1 - (1-r)\,t^{-1/(p-1)}\Bigr)\\
    \om''\,&=\,\Psi(r,t)\Bigl(2t^{-1/(p-1)} + (1-r)t^{-2/(p-1)}\Bigr).
\end{align*}
Thus we have for $\Delta^N_p\,\om = |\nabla \om|^{2-p}\Delta_p\,\om$
\begin{gather*}
  \Delta^N_p\,\om = (p-1)\,\om'' + \tfrac{n-1}{r}\,\om'\\ 
  =\Psi(r,t)\Bigl(-\tfrac{n-1}{r} + 2(p\!-\!1)t^{-1/(p-1)} + (1\!-\!r) t^{-1/(p-1)}\bigl[(p\!-\!1)t^{-1/(p-1)} -\tfrac{n-1}{r}\bigr]\Bigr).
\end{gather*}
Taking $t$ small enough, the positive terms dominate when $r > R$  and within the given bound for $\tau$ (so that the bracket above as well as the first two terms inside parenthesis can be estimated by $t^{-1/(p-1)}$ from below), we  easily obtain
$$ \Delta^N_p\,\om\,\geq\,\Psi(r,t)\Bigl(t^{-1/(p-1)} + (1-r)t^{-2/(p-1)}\Bigr)\,\geq\, \Psi(r,t)(1-r)t^{-2/(p-1)}.$$
From
$$\Bigl\vert\frac{\nabla \om}{\om}\Bigr\vert\,\geq\, \Bigl\vert\frac{1+(1-r)t^{-1/(p-1)}}{1-r}\Bigl \vert\,\geq\, t^{-1/(p-1)},$$
  we deduce
  $$\Bigl\vert\frac{\nabla \om}{\om}\Bigr\vert^{p-2}\Delta^N_p\,\om\, \geq\,
  \Psi(r,t) t^{-\frac{p-2}{p-1}}(1-r)t^{-2/(p-1)}\,=\, \Psi(r,t)(1-r)t^{-p/(p-1)}.$$

  On the other hand,
  \begin{align*}
 (p-1)\partial_t\,\om\,&=\, \Psi(r,t)(1-r)\,rt^{-p/(p-1)}\,\leq\,\, \Psi(r,t)(1-r)t^{-p/(p-1)}\\
&    \leq \Bigl\vert\frac{\nabla \om}{\om}\Bigr\vert^{p-2}\Delta^N_p\,\om\,=\, \om^{2-p}\Delta_p\,\om.
\end{align*}
  We conclude that the desired inequality $\partial_t(\om^{p-1})\leq \Delta_p\,\om$ is valid in\\ $(B_1\setminus\overline{B_R}) \times (0,\tau).$ --- We read off that $\Delta_p\,\om > 0.$
\end{proof}

    \noindent\textbf{Remark.} In $(B_{R_2}\backslash \overline{B_{R_1}}) \times (0,\tau),\,0 < R_1<R_2,$ the subsolution is
    $$\om_{R_1,R_2}\,=\, (R_2-|x|)\exp\left(-|x|\,R_2^{\frac{1}{p-1}}t^{-\frac{1}{p-1}}\right),$$
    where
    $$\tau\,=\,\Bigl(\tfrac{p-2}{n-1}\Bigr)^{p-1}R_1^{\,p-1}R_2\quad (n>1).$$
    This follows from Lemma \ref{omega} by replacing $x,t$ with $x/R_2,t/R_2^p$ and multiplying the barrier by $R_2$. When $n=1$, we may again take $\tau=\infty$.

\section{Weak Solutions are Viscosity Solutions}
We shall need a simple inequality.
\begin{lem}\label{delta} Let $p \geq 2$. If $|\phi| \leq 1,\,|\delta| \leq 1,$ then
  $$\bigl||\phi + \delta|^{p-2}-|\phi|^{p-2}\bigr|\,\leq\,C_p|\delta|^{\,\min\{1,p-2\}}.$$
\end{lem}

\begin{proof}
This follows easily from the infinitesimal calculus.
\end{proof}

The next theorem is valid also for sign-changing weak supersolutions.
\begin{thm}Let $p>2$. A continuous weak supersolution is a viscosity supersolution.
\end{thm}

\begin{proof}
Let $v$ be the weak supersolution. If it is \emph{not} a viscosity supersolution, there is a point $(x_0,t_0) \in \Omega_T$ and a test function $\phi \in C^2(\Omega_T)$ touching $v$ from below at $(x_0,t_0)$ and satisfying
$$\partial_t\bigl(|\phi(x_0,t_0)|^{p-2}\phi(x_0,t_0)\bigr) - \Delta_p\phi(x_0,t_0) = -\kappa < 0.$$ We have $\phi(x_0,t_0) = v(x_0,t_0)$ and
$$\phi(x,t) < v(x,t)\quad\text{when}\quad (x,t) \neq (x_0,t_0).$$ By continuity,
$$\partial_t\bigl(|\phi(x,t)|^{p-2}\phi(x,t)\bigr) - \Delta_p\phi(x,t) \,<\,-\frac{\kappa}{2}$$
  in a neighbourhood $|x-x_0| < r,\, |t-t_0| < \gamma$, say in $U_{r,\gamma}.$ Upon multiplying $v$ and $\phi$ by the same constant, we may assume that $|\phi| <1$.

  Define
  $$\delta \,=\,\frac{1}{2}\,\underset{\partial_pU'}{\inf}\{v-\phi\},\quad U' = U_{r',\gamma'}\subset U_{r,\gamma}$$  
  and consider the translated function $\phi_{\delta} = \phi + \delta,$ when $\delta > 0$ is a small constant. We claim that  $\phi_{\delta}$ itself is a weak subsolution.
  To see this, calculate
  \begin{align*}
    &|\partial_t\bigl(|\phi_{\delta}|^{p-2}\phi_{\delta} - |\phi|^{p-2}\phi\bigr)|\\
   &\quad =\, (p-1)|(|\phi + \delta|^{p-2}-|\phi|^{p-2})\partial_t\phi|\\
   & \quad\leq\,(p-1)|\partial_t\phi|C_p\delta^{\,\min\{1,p-2\}}.
  \end{align*}
  Here Lemma \ref{delta} was used.
  Since $\Delta_p(\phi+\delta) =\Delta_p\phi$ we obtain
  \begin{align*}
    &\partial_t\bigl(|\phi_{\delta}|^{p-2}\phi_{\delta}\bigr) - \Delta_p\phi_{\delta}\\
    &\leq \, \partial_t\bigl(|\phi|^{p-2}\phi\bigr) - \Delta_p\phi + C_p\|\partial_t\phi\|\delta^{\,\min\{1,p-2\}}\\
   & < \,-\frac{\kappa}{2} + O\bigl(\delta^{\,\min\{1,p-2\}}\bigr) < -\frac{\kappa}{4} < 0
  \end{align*}
  when $\delta <$ some $\delta_{\kappa}.$ To adjust $\delta$, decrease $U'$. This proves that $\phi_{\delta}$ is a (classical) subsolution for $\delta$ small enough, i.e.\ in a sufficiently small neighbourhood $U'$ of $(x_0,t_0)$.

  To continue, we intend to apply the Comparison Principle, Theorem~\ref{gap},  on $\phi_{\delta}$ and $v$. On the parabolic boundary $\partial_p U'$, we have a gap:
  $$\phi_{\delta} = \phi + \delta \leq \phi + \frac{v-\phi}{2} = \frac{v+\phi}{2}\leq v-\delta < v.$$ By comparison, $$\phi_{\delta} \leq v$$ in the whole $U'$. But at the point $(x_0,t_0)$, this yields the contradiction
  $$\phi(x_0,t_0) + \delta \leq v(x_0,t_0) = \phi(x_0,t_0).$$
  Therefore, such a test function cannot exist. This proves that $v$ is a viscosity supersolution.
\end{proof}

  The proof is similar for weak subsolutions. Theorem \ref{thmII} follows.

  \section{Viscosity Solutions}

  A viscosity solution $u$ is even a weak solution in the open set where $u\neq 0.$ See \cite{LPS} for this basic result. Thus it is crucial to investigate (the boundary of) the zero set $\{u=0\}.$ As we shall see, for \emph{non-negative} viscosity solutions, the zero set is either empty, all of $\Omega_T$, or $\Omega\times(0,t_0]$ for some $0<t_0<T$. This is already known for weak solutions but the presently known versions of the Harnack Inequality do not directly apply for viscosity solutions. A more sophisticated arrangement is called for.

  \paragraph{Non-negative Viscosity Subsolutions.} If $u \geq 0$ is a viscosity subsolution, so is
  $$u_{\varepsilon}\,=\,\max\{u,\varepsilon\},\quad \varepsilon > 0,$$
  as a direct consequence of the viscosity definition.
  Since $\inf u_{\varepsilon}\geq\varepsilon>0$, it is even a weak subsolution: the proof of \cite[Theorem~2.8]{LPS} applies also to subsolutions, using sup-convolution in place of inf-convolution. The transmission to $u$ requires some estimates.
  \begin{lem}\label{subenergy} If $u \in L^p_{loc}\bigl(0,T;W^{1,p}_{loc}(\Omega)\bigr)$ is a nonnegative weak subsolution, then
    \begin{align*}
      \int\limits_{t_1}^{t_2}\!\!\!\int\limits_{\Omega}\zeta^p|\nabla u|^p\d x \d t\,\leq\,& p^p \int\limits_{t_1}^{t_2}\!\!\!\int\limits_{\Omega}u^p|\nabla \zeta|^p \d x \d t\\
      &+ (p-1) \bigg\vert\,\bigg\vert_{t_1}^{t_2}\int\limits_{\Omega} \zeta^pu^p\d x\bigg\vert + (p-1)\int\limits_{t_1}^{t_2}\!\!\!\int\limits_{\Omega}u^p\big\vert\frac{\partial\,}{\partial t}\zeta^p\big\vert \d x\d t
    \end{align*}
    for $0 < t_1<t_2 <T$ and for all test functions $\zeta = \zeta(x,t) \geq 0$ which are compactly supported in $x$. Especially, if $0\leq u\leq M$ in $\Omega_T$, then  $u$ belongs to $L^p\bigl(0,T;W^{1,p}_{loc}(\Omega)\bigr)$ and if $\zeta = \zeta(x),\,0\leq\zeta\leq 1,$
then    $$\int\limits_{0}^{T}\!\!\!\int\limits_{\Omega}\zeta^p|\nabla u|^p\d x \d t\,\leq C M^p\Bigl(T\!\int\limits_{\Omega}|\nabla \zeta|^p \d x + 1 \Bigr).$$
     \end{lem}

  \begin{proof}
  Formally, the choice $\phi = \zeta^p u$ in
  \begin{align*}
    \bigg\vert_{t_1}^{t_2}\int\limits_{\Omega}\phi u^{p-1}\d x - \int\limits_{t_1}^{t_2}\!\!\!\int\limits_{\Omega}u^{p-1}\frac{\partial\phi}{\partial t} \d x\d t
    + \int\limits_{t_1}^{t_2}\!\!\!\int\limits_{\Omega}|\nabla u |^{p-2}\langle\nabla u,\nabla \phi\rangle \d x \d t\,\leq\,0
  \end{align*} yields the estimate. A proper proof goes via the Stekloff averages in the regularized equation
  \begin{align*}\bigg\vert_{t_1}^{t_2}\int\limits_{\Omega}\phi [ u^{p-1}]_h\d x -\int\limits_{t_1}^{t_2}\!\!\!\int\limits_{\Omega}[u^{p-1}]_h\frac{\partial\phi}{\partial t} \d x\d t\le -
    \int\limits_{t_1}^{t_2}\!\!\!\int\limits_{\Omega}\langle[|\nabla u |^{p-2}\nabla u]_h,\nabla \phi\rangle \d x \d t.
  \end{align*}
  Write
  $$\beta= u^{p-1},\quad \beta_h = [u^{p-1}]_h;\quad u = \beta^{q-1},\quad q = \frac{p}{p-1}.$$
  Use $\phi = \zeta^p\beta_h^{q-1}$ (justified by first replacing
  $\beta_h$ with $\beta_h+\delta$ in the test function and letting
  $\delta\to0$ for fixed $h$\footnote{The spatial gradient satisfies
  $|\nabla((\beta_h+\delta)^{q-1})|^p\le [|\nabla u|^p]_h$
  by the Sobolev chain rule, H\"older's inequality and convexity.
  For fixed $h,\delta>0$, the time derivative of $(\beta_h+\delta)^{q-1}$ is locally bounded
  by the usual chain rule, the formula for the derivative of a Stekloff average,
  and the local boundedness of $u$.
  In the sign-changing case such test functions can be problematic.}),
   and note that $\beta_h^{q-1} \rightarrow u$ as $h \to 0$. The left-hand side (the part with time derivatives) can be integrated to
  \begin{align*}
    \mathrm{LHS} =\,& \frac{1}{q}\, \bigg\vert_{t_1}^{t_2}\int\limits_{\Omega}\zeta^p\beta_h^q\d x -\frac{1}{q}\int\limits_{t_1}^{t_2}\!\!\!\int\limits_{\Omega}\beta_h^q\frac{\partial\,}{\partial t}\zeta^p \d x\d t\\
    \underset{h \to 0}{\longrightarrow}\,& \frac{1}{q}\, \bigg\vert_{t_1}^{t_2}\int\limits_{\Omega}\zeta^pu^p\d x -\frac{1}{q}\int\limits_{t_1}^{t_2}\!\!\!\int\limits_{\Omega} u^p\frac{\partial\,}{\partial t}\zeta^p \d x\d t.
  \end{align*}
  The elliptic part is harmless:
  \begin{align*}
   \mathrm{RHS}\,=\,&- \int\limits_{t_1}^{t_2}\!\!\!\int\limits_{\Omega}\langle[\,|\nabla u |^{p-2}\nabla u]_h,\nabla (\zeta^p\beta_h^{q-1})\rangle \d x \d t \\
    \underset{h \to 0}{\longrightarrow}\,& -\int\limits_{t_1}^{t_2}\!\!\!\int\limits_{\Omega}\langle|\nabla u |^{p-2}\nabla u,\nabla (\zeta^pu)\rangle\, \d x \d t .
  \end{align*}
  Then use Young's Inequality
  $$\big\vert p\zeta^{p-1}u\langle|\nabla u|^{p-2}\nabla u,\nabla \zeta \rangle\big\vert\,\leq\, \frac{1}{q}\,\zeta^p|\nabla u|^p + \frac{p^p}{p}|u|^p|\nabla \zeta|^p$$
  to conclude the proof.
\end{proof}

As stated before the lemma, for a viscosity subsolution $u$ it holds that  $u_{\varepsilon} = \max\{u,\varepsilon\}$, $\eps>0$, is also a weak subsolution.
  Fix $0<t_1<t_2<T$ and let us proceed to the limit $\varepsilon \to 0$. By  Lemma \ref{subenergy} applied on $u_{\varepsilon} = \max\{u,\varepsilon\}$ in place of $u$,
    $$\int\limits_{t_1}^{t_2}\!\!\!\int\limits_{\Omega}\zeta^p|\nabla u_{\varepsilon} |^p\, \d x \d t \,\leq\,Q\qquad(0 < \varepsilon <1)$$
    with $Q$ independent of $\varepsilon$.  Hence
    $$\int\limits_{t_1}^{t_2}\!\!\!\int\limits_{\Omega}\zeta^p|\nabla u |^p\, \d x \d t 
    \,\leq\,Q.$$
    Moreover,
    \begin{align*}\int\limits_{t_1}^{t_2}\!\!\!\int\limits_{\Omega}\zeta^p|\nabla u -\nabla u_{\varepsilon}  |^p\, \d x \d t\,=\underset{\{0\leq u\leq\varepsilon\}}{\int\!\!\!\int}\!\zeta^p|\nabla u |^p\, \d x \d t 
      \underset{\varepsilon \to 0}{\longrightarrow} \underset{\{u=0\}}{\int\!\!\!\int}\zeta^p|\nabla u|^p \,\d x \d t
    \end{align*}
    by Lebesgue's theorem. The last integral is zero, because $\nabla u = 0$ a.e. in the set $\{u=0\}.$

    Now we can pass to the limit in
    \begin{align*}
    \bigg\vert_{t_1}^{t_2}\int\limits_{\Omega}\zeta u_{\varepsilon}^{p-1} \, \d x\,-\,
    \int\limits_{t_1}^{t_2}\!\!\!\int\limits_{\Omega} u_{\varepsilon}^{p-1}\frac{\partial \zeta }{\partial t} \, \d x \d t
    +  \int\limits_{t_1}^{t_2}\!\!\!\int\limits_{\Omega}\langle |\nabla u_{\varepsilon} |^{p-2}\nabla u_{\varepsilon}, \nabla \zeta \rangle\,  \d x \d t \, \leq \,0
    \end{align*}
    to arrive at the equation (\ref{zetaeq}) valid for $u$ itself. Thus we have obtained the following theorem.

    \begin{thm}\label{viscsubweak} Let $p \geq 2$. A continuous viscosity subsolution $u \geq 0$ in $\Omega_T$ is also a weak subsolution.  If it is bounded, it belongs to $L^p\bigl(0,T;W^{1,p}_{loc}(\Omega)\bigr)$.
    \end{thm}

    An immediate consequence of the definition is that, if $u_1$ and $u_2$ are viscosity subsolutions, so is the function $\max\{u_1,u_2\}.$ Therefore we can include a corollary.

    \begin{cor} Suppose that  $u_1$ and $u_2$ are viscosity subsolutions in $\Omega_T$. So is $\max\{u_1,u_2\}$,
      and if $\max\{u_1,u_2\}\,\geq\,0$, then it is also a weak subsolution.
    \end{cor}

    \medskip
    \paragraph{Non-negative Viscosity Supersolutions.} This case is harder. Using the result in \cite{LPS} again, we already have that a continuous viscosity supersolution $v$ is a weak supersolution in the open set $\{v > 0\}$.

    We begin with a standard estimate.

    \begin{lem}\label{superM} Suppose that $v \in L^p_{loc}\bigl(0,T;W^{1,p}_{loc}(\Omega)\bigr)$ is a weak supersolution in $\Omega_T$ and $0 < v(x,t) \leq M.$ Then
      \begin{align*} \int\limits_{t_1}^{t_2}\!\!\!\int\limits_{\Omega} \zeta^p|\nabla v |^{p}  \d x \d t \,\leq &\, p^{\,p}M^p\int\limits_{t_1}^{t_2}\!\!\!\int\limits_{\Omega} |\nabla \zeta |^{p}  \d x \d t\\        &+\,2pM^p\,\int\limits_{t_1}^{t_2}\!\!\!\int\limits_{\Omega}\Bigl\vert\frac{\partial \zeta^p}{\partial t}\Bigr\vert\,\d x \d t\,+\,4pM^p\,|\Omega|
      \end{align*}
      where $0\le \zeta\le 1$ is a test function compactly supported in the $x$-variable. In particular, $v$ belongs to $ L^p\bigl(0,T;W^{1,p}_{loc}(\Omega)\bigr).$
    \end{lem}

    \begin{proof}
    Use the test function $\phi = \zeta^p(M-v)$ in the equation
    \begin{align*}
      \bigg\vert_{t_1}^{t_2}\int\limits_{\Omega}v^{p-1}\phi\,\d x\,-\,\int\limits_{t_1}^{t_2}\!\!\!\int\limits_{\Omega} v^{p-1}\frac{\partial \phi}{\partial t}\,  \d x \d t
      +\, \int\limits_{t_1}^{t_2}\!\!\!\int\limits_{\Omega}\langle|\nabla v|^{p-2}\nabla v,\nabla \phi \rangle \,  \d x \d t\,\geq \,0.
    \end{align*}

    A proper proof  can be based on the test function
    \begin{align*}\phi\,=\,\zeta^p\bigl(M-\beta_h^{q-1}\bigr),\quad \frac{1}{p}+\frac{1}{q}\,=\,1 \\
      \beta\,=\,v^{p-1},\quad \beta_h\,=\, [v^{p-1}]_h =\quad \text{Stekloff average},
    \end{align*}
    inserted into the regularized equation
    \begin{align*}
      \bigg\vert_{t_1}^{t_2}\int\limits_{\Omega}[v^{p-1}]_h\phi\,\d x\,-\,\int\limits_{t_1}^{t_2}\!\!\!\int\limits_{\Omega} [v^{p-1}]_h\frac{\partial \phi}{\partial t}\,  \d x \d t
      +\, \int\limits_{t_1}^{t_2}\!\!\!\int\limits_{\Omega}\langle[|\nabla v|^{p-2}\nabla v]_h,\nabla \phi \rangle \,  \d x \d t\,\geq \,0.
    \end{align*}
    Since the constants are irrelevant and  the standard proof is pretty similar to that of Lemma \ref{subenergy}, we omit the details.
\end{proof}

    The next lemma makes it possible to estimate a \emph{viscosity} supersolution from below. An explicit, and this is crucial, minorant $\omega$ is provided by Lemma \ref{omega} and the remark after it.

    \begin{lem} [$x$-\textsf{expansion}]\label{xexpansion} Let $v\geq 0$ be a viscosity supersolution in $B_1 \times (0,1)$ and suppose that $ v(x,t) > 0$ in $B_{2R}\times (0,\tau)$ for some $R < \frac{1}{2}$ and $0<\tau<	1.$ Assume that $\omega \geq 0$ is a smooth bounded weak subsolution in $(B_1\setminus \overline{B}_R)\times (0,\tau)$ such that
      \begin{equation*}
        \begin{cases}
          \omega < v\quad\text{on}\quad \partial B_R \times [0,\tau)\\
            \omega \leq v\quad\text{on}\quad  (B_1\setminus \overline{B}_R) \times \{0\}\\
            \omega = 0 \quad\text{on}\quad \partial B_1 \times [0,\tau)\\
              \Delta_p\omega \,>\,0 \quad\text{in}\quad (B_1\setminus \overline{B}_R)\times (0,\tau).
        \end{cases}\end{equation*}
    Then $\omega\leq v$ in $(B_1\setminus \overline{B}_R)\times (0,\tau).$
      \end{lem}

      \begin{proof}
      For $\e >0$,  we consider
      $$\omega_{\e}(x,t)\,=\,e^{- \e/(\tau-t)}\omega(x,t).$$ 
      Clearly, $\Delta_p\,\omega \,>\,0.$ Since $\partial_t\omega^{p-1} \leq \Delta_p\omega,$
      \begin{equation}\label{eqna}\partial_t\omega_{\e}^{p-1}\, \leq\,\Delta_p\,\omega_{\e} - \e\,\frac{p-1}{(\tau-t)^2}\,\omega_{\e}(x,t)^{p-1}
      \end{equation}
      in  $(B_1\setminus \overline{B}_R)\times (0,\tau).$

      It is sufficient to show that $v \geq \omega_{\e}$ for all small $\e > 0$. Suppose now, on the contrary, that the difference $\omega_{\e}-v$ attains a \emph{strictly positive}
      maximum $c = c_{\e}$ at the point
      $$(\hat{x},\hat{t}) \in \overline{B_1\setminus B_R} \times [0,\tau].$$
Taking into account the boundary conditions and $\om_\eps\to 0$ as $t\to \tau$,  it holds that $(\hat{x},\hat{t})\in (B_1\setminus \overline B_R) \times (0,\tau)$ and inequality (\ref{eqna}) remains valid there.

      Now
      $$\omega_{\e}-v\,\leq \, c\,=\, \omega_{\e}(\hat{x},\hat{t})-v(\hat{x},\hat{t})$$
      and so the definition of viscosity supersolutions requires that
      $$\partial_t\left(|\omega_\varepsilon-c|^{p-2}(\omega_\varepsilon-c)\right)=(p-1)|\omega_\varepsilon-c|^{p-2}\partial_t\omega_\varepsilon
\ge \Delta_p\omega_\varepsilon \quad\text{at}\quad(\hat{x},\hat{t}).$$
      Since $\omega_{\e}-c = v \geq 0$ at the point $(\hat{x},\hat{ t}),$ 
      \begin{equation}\label{eqnb}        (p-1)v(\hat{x},\hat{t})^{p-2}\partial_t\omega_{\e}(\hat{x},\hat{t})\,\geq\,\Delta_p\,\omega_{\e}(\hat{x},\hat{t}).
      \end{equation}
      On the other hand, (\ref{eqna}) yields
      $$ (p-1)\omega_{\e}(\hat{x},\hat{t})^{p-2}\partial_t\omega_{\e}(\hat{x},\hat{t})\,<\, \Delta_p\,\omega_{\e}(\hat{x},\hat{t}),$$
          since $$\omega_{\e}(\hat{x},\hat{t}) >v(\hat{x},\hat{t})\geq 0.$$ It follows that
          \begin{equation}\label{eqnc}          \partial_t\omega_{\e}(\hat{x},\hat{t})\, <\,\frac{1}{(p-1)\omega_{\e}(\hat{x},\hat{t})^{p-2}}\, \Delta_p\,\omega_{\e}(\hat{x},\hat{t}).
          \end{equation}
          Since
          $$\Delta_p \,\omega_{\e}\, =\, \exp\left(-\frac{\e(p-1)}{\tau-t}\right)\,\Delta_p\, \omega > 0$$
          by assumption, inequality (\ref{eqnb}) implies $v(\hat{x},\hat{t}) > 0.$ Combining (\ref{eqnb}) and (\ref{eqnc}) we end up with
          \begin{align*}
            \Delta_p\,\omega_{\e}(\hat{x},\hat{t})\,\leq&\, (p-1)v(\hat{x},\hat{t})^{p-2}\partial_t\omega_{\e}(\hat{x},\hat{t})\\
            <&\, \frac{v(\hat{x},\hat{t})^{p-2}}{\omega_{\e}(\hat{x},\hat{t})^{p-2}}\,\Delta_p\,\omega_{\e}(\hat{x},\hat{t}) .\end{align*}
          This is a contradiction, because $v(\hat{x},\hat{t}) < \omega_{\e}(\hat{x},\hat{t})$ and so the fraction is $< 1$. Therefore $\omega_{\e} -v \leq 0.$
\end{proof}

          \begin{cor} [\textsf{Expansion to $\Omega \times \{t_0\}$}] \label{horizontal} Let $p > 2$. Let $v\geq 0$ be a continuous viscosity supersolution in $\Omega_T$. If $v(x_0,t_0) > 0$ at some interior point $(x_0,t_0) \in \Om_T$, then $v(x,t_0) > 0$ for all $x \in \Om.$
          \end{cor}

          \begin{proof}
          If $v(x_0,t_0) > 0$, then $v(x,t) > 0$ in $\overline{B(x_0,R_1)} \times (t_0-\delta,t_0 + \delta),$ say. Inserting  a suitable barrier $ \lambda \omega_{R_1,R_2}$ from Lemma \ref{omega} into  Lemma \ref{xexpansion}, we see that $v(x,t) > 0$ in $B(x_0,R_2) \times (t_0-h,t_0+h)$ for some $h > 0$, if $\overline{B(x_0,R_2)} \subset \Om.$ It follows  that $v(x,t_0) > 0$ in $B(x_0,\dist(x_0,\partial \Om)).$ Thus the positivity has expanded in the $x$-direction at time $t_0$. 
          
          We shall iterate the procedure along a standard chain of overlapping balls to extend it to the whole domain. Fix an arbitrary point $y \in \Om.$ We claim that $v(y,t_0) > 0.$ To see this, let $\gamma \subset \Om$ be a path joining the points $x_0$ and $y.$ Let $4 \rho = \dist(\gamma, \partial \Om).$  We construct a suitable chain of finitely many pairwise overlapping balls:
          \begin{equation*}
            \gamma\, \subset\, \overset{N}{\underset{i=1}{\bigcup}}B(x_i,\rho),\quad x_i \in \gamma.
          \end{equation*}
          If we have reached that $v(x_i,t_0) > 0$, then $v(x,t_0) > 0$ in $B(x_i,4\rho)$. Certainly the next center, say $x_{i+1}$ belongs to this enlarged ball and so $v(x_{i+1},t_0) > 0$. Repeating the procedure we finally reach the point $y$.
\end{proof}  

          It follows that the zero set is of the form
          $$\{v=0\}\,=\, \underset{\iota}{\bigcup}\,\Om \times I_{\iota},$$
          where the $I_{\iota}$ are closed  time intervals or points, and so the positivity set
          $$\{v>0\}\,=\, \underset{\iota}{\bigcup}\,\Om \times (t^-_{\iota},t^+_{\iota})$$
          with open time intervals.
 This structure is decisive for the use of the barrier in the next lemma.

          \begin{lem}[\textsf{$t$-expansion}]\label{time}
            Let $v\geq 0$ be a continuous viscosity supersolution in $\Omega_T$.
            If $v(x_0,t_0)>0$ at some $(x_0,t_0)\in\Omega_T$, then
            $v(x_0,t)>0$ when $t_0\leq t<T$.
          \end{lem}

          \begin{proof}
          By the previous corollary, $v(x,t_0) > 0$ when $x \in \Om$.

          Let $B(x_0,R) \subset \subset \Om$ and choose a suitable barrier
          $$V(r,t)\,=\, C\,e^{-\frac{\lambda(t-t_0)}{p-1}}\psi_{\lambda}(r),\qquad r = |x-x_0|$$
          so that $v(x,t_0) > V(r,t_0) > 0,\ r =|x-x_0|,$  when $x \in B(x_0,R).$ Here $\psi_\lambda$ is the positive first eigenfunction in (\ref{eigenfunction}), and $\lambda = \lambda(R) > 0$ is its eigenvalue. On $\partial B(x_0,R) \times [t_0,\infty)$, we have $V=0$. In the set where $v > 0$, we know that $v$ is a weak supersolution. Hence the Comparison Principle (with a gap) in Theorem \ref{gap} is valid in $B(x_0,R) \times [t_0,t_1]$, provided that this space$\times$time cylinder is comprised in the positivity set, and we will choose next the cylinder in such a way that it is.

            If it so happens that $v(x_0,t) = 0$ at some time $t > t_0$, let $t^{\star}$ be the first zero ( $ t^{\star} =\inf\{t > t_0|\, v(x_0,t) = 0\}$). By the previous considerations, $v > 0$ in $\Om \times [t_0,t^{\star})$ (and especially in $\overline{B(x_0,R)} \times [t_0,t^{\star})$).
              Therefore the Comparison Principle  in Theorem \ref{gap} is applicable here and so  $v \geq V$ in $B(x_0,R) \times [t_0,t^{\star}-\delta]$, for all small $\delta > 0$, and so in $B(x_0,R) \times [t_0,t^{\star})$. By continuity, $v(x_0,t^{\star}) \geq V(0,t^{\star})> 0.$ This strict inequality is a contradiction. We conclude that $v(x_0,t) > 0$ when $t_0 \leq t <T$.
\end{proof}

                  Excluding the case $v \equiv 0$ we have reduced the problem with non-negative viscosity
                  supersolutions to the following situation. For some $t_0 \in [0,T)$,
                    \begin{equation*}\begin{cases}
                        v(x,t) = 0\quad\text{if}\quad x\in \Om,\quad 0<t \leq t_0,\\
                        
                        v(x,t) > 0\quad\text{if}\quad x\in \Om,\quad t_0 < t < T. \end{cases}
                    \end{equation*}
                    It follows that $v$ is also a weak supersolution except possibly on one  time slice $t=t_0$.

                    \begin{cor}[\textsf{Zero extension to the past}]\label{past}
                   Suppose that $v$ is bounded and continuous in $\Omega_T$,
                    that $v=0$ in $\Omega\times(0,t_0]$, and that $v>0$ in $\Omega\times(t_0,T)$.
                    If $v$ is a weak supersolution in $\Omega\times(t_0,T)$,
                    then $v$ is a weak supersolution in $\Omega_T$.
                    Furthermore, $v\in L^p\bigl(0,T;W^{1,p}_{loc}(\Omega)\bigr)$.
                    \end{cor}

                    \begin{proof}
                                     By Lemma \ref{superM}, $v$ belongs to $L^p\bigl(t_0,T;W^{1,p}_{loc}(\Om)\bigr)$.
                    Hence $v \in L^p\bigl(0,T;W^{1,p}_{loc}(\Om)\bigr)$. Fix a test function $\phi \geq 0, \, \phi \in C_0^{\infty}(\Omega_T)$, and define
                    \begin{align*}
                      \eta(t)\,=\,
                      \begin{cases}
                      1,\quad t \geq t_0 + \e\\
                      \frac{t-t_0}{\e},\quad t_0 <t < t_0 +\e\\
                      0,\quad t \leq t_0 \end{cases}
                    \end{align*}
                    for some $\e > 0$\footnote{We omit for simplicity two approximation procedures, as quite common in the literature, the smooth approximation of $\eta$, and its definition first for $t_1>t_0$, so that the test function would be in $ C_0^{\infty}(\Om\times (t_0,T))$.}. Then we see that
                    \begin{align*}
                      0\,&\geq\,\int\limits_{0}^{T}\!\!\!\int\limits_\Om \Bigl(v^{p-1}\frac{\partial\,}{\partial t}\,(\eta\phi)-\langle|\nabla v|^{p-2}\nabla v,\nabla(\eta \phi)\rangle\Bigr)\,\d x \d t\\
                      &=\int\limits_{0}^{T}\!\!\!\int\limits_\Om \eta \Bigl(v^{p-1}\frac{\partial\phi}{\partial t}-\langle|\nabla v|^{p-2}\nabla v,\nabla \phi\rangle\Bigr)\,\d x \d t + \frac{1}{\e} \int\limits_{t_0}^{t_0+\e}\!\!\!\int\limits_\Om\,v^{p-1}\phi \d x \d t\\
                      &\hspace{-1 em}\underset{\e \to 0}{\longrightarrow}\,
                      \int\limits_{0}^{T}\!\!\!\int\limits_\Om  \Bigl(v^{p-1}\frac{\partial\phi}{\partial t}-\langle|\nabla v|^{p-2}\nabla v,\nabla \phi\rangle\Bigr)\,\d x \d t   +  \int\limits_\Om v(x,t_0)^{p-1}\phi(x,t_0)\, \d x.
                    \end{align*}
                    Since the very last integral is zero (continuity of $v$ was used), we have arrived at the desired differential inequality, from which we can read off that $v$ is a weak supersolution in the whole $\Om_T$.
\end{proof}

Thus we can extend the weak supersolution property of viscosity
supersolutions across $t=t_0$, resolving the remaining issue noted
after Lemma \ref{time}. We have proved the following.
                    \begin{samepage}
                    \begin{prop} Let $p > 2$. A continuous viscosity supersolution $v \geq 0$ is also a weak supersolution. If it is bounded, then
                      $v \in L^p\bigl(0,T;W^{1,p}_{loc}(\Om)\bigr)$.
                    \end{prop}
\par\end{samepage}

Given the subsolution result in Theorem \ref{viscsubweak},
this proves Theorem~\ref{thmI}.

\medskip
\noindent\textbf{Remark.} The following applications are immediate.
Theorems \ref{thmI} and \ref{thmII} give the equivalence of weak and viscosity solutions in the non-negative case. Thus non-negative viscosity solutions are locally H\"older continuous and satisfy Harnack's Inequality, since these properties hold for weak solutions by \cite[Theorems~1.1 and~B.1]{BDL}. For the Harnack Inequality for positive weak solutions, see also \cite{KK}, \cite{GV}, and  \cite{IMM}.

  \vspace{1 ex}
{\small \textsf{Acknowledgements}: MP acknowledges support from the Research Council of Finland, project 360185.  
}

\bigskip
\small{

\begin{samepage}
\noindent \textsf{Peter Lindqvist\\ Department of
   Mathematical Sciences\\ Norwegian University of Science and
  Technology\\ N--7491 Trondheim, Norway}\\
\textsf{e-mail}: peter.lindqvist@ntnu.no
\par\end{samepage}

 \bigskip
\begin{samepage}
\noindent {\textsf{Mikko Parviainen \\  Department of Mathematics and Statistics\\ University of Jyv\"askyl\"a \\ 
FI-40014 Jyv\"askyl\"a, Finland}  \\
  \textsf{e-mail}: mikko.j.parviainen@jyu.fi}
\par\end{samepage}
  
\bigskip
\begin{samepage}
\noindent {\textsf{Jarkko Siltakoski \\  Department of Mathematics and Statistics\\  University of Jyv\"askyl\"a\\
FI-40014 Jyv\"askyl\"a, Finland}  \\
  \textsf{e-mail}: jarkko.siltakoski@jyu.fi }
\par\end{samepage}

}

\end{document}